\documentclass[oneside,reqno,10pt]{amsart}
\usepackage{amsfonts,amssymb,amsmath,amsthm}

\usepackage{hyperref}
\usepackage{centernot}
\usepackage{mathtools}
\usepackage{ stmaryrd }
\usepackage{graphicx}
\usepackage{euscript,enumerate}
\usepackage{color,wrapfig}
\usepackage{pxfonts}
\usepackage{verbatim}

\theoremstyle{plain} 
\newtheorem{theorem}             {Theorem} 
\newtheorem{corollary}  [theorem]{Corollary}

\newtheorem{lemma}[theorem]{Lemma}
\newtheorem{conjecture}[theorem]{Conjecture}

\theoremstyle{definition}
\newtheorem*{definition}{Definition}

\newtheorem{problem}{Problem}

\theoremstyle{remark}

\DeclareMathOperator{\lcm}{lcm}
\def\modd#1 #2{#1\ \mbox{\rm (mod}\ #2\mbox{\rm )}}

\begin{document}

\author{Daniel Tsai}
\address{Nagoya University, Graduate School of Mathematics, 464-8602, Furocho, Chikusa-ku, Nagoya, Japan}
\email{shokuns@math.nagoya-u.ac.jp}

\title{$v$-Palindromes: An Analogy to the Palindromes}

\begin{abstract}
In the year 2007, the author discovered an intriguing property of the number $198$ he saw on the license plate of a car. Namely, if we take $198$ and its reversal $891$, prime factorize each number, and sum the numbers appearing in each factorization, both sums are $18$. Such numbers are formally introduced in a short published note in 2018. These numbers are later named $v$-palindromes because they can be viewed as an analogy to the usual palindromes. In this article, we introduce the concept of a $v$-palindrome in base $b$, and prove their existence for infinitely many bases. Finally, we collect some conjectures on $v$-palindromes.
\end{abstract}
\maketitle

\section{Origin of the concept.}

As I recall, it was some time in the first half of the year 2007 when I was 15 years old. My mother and younger brother were in a video rental shop near our home in Taipei and my father and I were waiting outside the shop, standing beside our parked car. I was a bit bored and glanced at the license plate of our car, which was 0198-QB. For no clear reason, I took the number $198$ and did the following. I factorized $198=2\cdot 3^2\cdot 11$, reversed the digits of $198$, and factorized $891=3^4\cdot 11$. Then I summed the numbers appearing in each factorization: $2+3+2+11=18$ and $3+4+11=18$ respectively. So surprisingly they are equal! We also illustrate this pictorially as follows.
\begin{align}
198\quad&=\quad 2\cdot 3^2\cdot 11\quad\xmapsto{\phantom{aaaa}}\quad 2+(3+2)+11 \nonumber \\
\uparrow\phantom{!}\quad&\phantom{=}\quad \phantom{2\cdot 3^2\cdot 11}\quad\phantom{\xmapsto{\phantom{aaaa}}}\quad \phantom{2+(3+2}\parallel\nonumber\\
\text{digit reversal}&\phantom{=}\quad \phantom{2\cdot 3^2\cdot 11}\quad\phantom{\xmapsto{\phantom{aaaa}}}\quad \phantom{2+(3+2}18 \label{picture} \\
\downarrow\phantom{!}\quad&\phantom{=}\quad \phantom{2\cdot 3^2\cdot 11}\quad\phantom{\xmapsto{\phantom{aaaa}}}\quad \phantom{2+(3+2}\parallel\nonumber\\
891\quad&=\quad \phantom{2\cdot !!} 3^4\cdot 11\quad\xmapsto{\phantom{aaaa}}\quad (3+4)+11 \nonumber
\end{align}
Afterwards, I spent some time to try to show that there is an infinitude of such numbers (we define them rigorously in the next section) but could not show it.

In October 2018, I published a one-and-a-half page note \cite{tsai0} in the S\={u}gaku Seminar magazine, which is sort of like the American Mathematical Monthly of Japan, merely defining such numbers and also showing their infinitude, although I recall already knowing how to show their infinitude as early as 2015.

In March 2021, I published the paper \cite{tsai1}, proving a general theorem pertaining to such numbers, and then in June 2021, uploaded the third version of the preprint \cite{tsai2}, in which more in-depth investigations were done. In this article, we re-introduce such numbers again as well as stating some results in \cite{tsai1,tsai2}.

\section{Definition of the property.}

Recall the base $b$ representation of a natural number, where $b\geq2$ is the base. For every natural number $n$, there exist unique integers $L\geq1$ and $0\leq a_0,a_1,\ldots,a_{L-1}<b$ with $a_{L-1}\neq0$ such that
\begin{equation}\label{baseb}
  n=a_0+a_1b+\cdots+a_{L-1}b^{L-1}.
\end{equation}
Thus $L$ is the number of base $b$ digits of $n$. We also denote \eqref{baseb} as $n=(a_{L-1}\cdots a_1a_0)_b$. We define the \textit{digit reversal} in base $b$ of $n$ to be
\begin{equation}
r_b(n)=a_{L-1}+a_{L-2}b+\cdots+a_{0}b^{L-1}.
\end{equation}
For instance, $r_{10}(18)=81$ and $r_{10}(2)=r_{10}(200)=2$. We next define a function $v(n)$ to denote ``summing the numbers appearing in the factorization''.
\begin{definition}
Suppose that the prime factorization of the natural number $n$ is
\begin{equation}
  n=p^{\varepsilon_1}_1\cdots p^{\varepsilon_s}_s q_1\cdots q_t,
\end{equation}
where the $s,t\geq0$ and $\varepsilon_1,\ldots,\varepsilon_s\geq2$ are integers and the $p_1,\ldots,p_s,q_1,\ldots,q_t$ are distinct primes, then we set
\begin{equation}
  v(n)=\sum^s_{i=1}(p_i+\varepsilon_i)+\sum^t_{j=1}q_j.
\end{equation}
\end{definition}
Notice that $v(n)$ is an additive function, i.e., $v(mn)=v(m)+v(n)$ whenever $m$ and $n$ are relatively prime natural numbers. The values of $v(n)$ has been created as sequence A338038 in the On-Line Encyclopedia of Integer Sequences \cite{oeis}. We can now make the following definition.
\begin{definition}[$v$-palindrome]
Let $b\geq2$ be an integer. A natural number $n$ is a $v$-\textit{palindrome} in base $b$ if
\begin{itemize}
\item[{\rm (i)}] $b\nmid n$,
\item[{\rm (ii)}] $n\neq r_b(n)$, and
\item[{\rm (iii)}] $v(n)=v(r_b(n))$.
\end{itemize}
The set of $v$-palindromes in base $b$ is denoted by $\mathbb{V}_b$.
\end{definition}
The condition (i) is included merely for the aesthetic look of $n$ and $r_b(n)$ having the same number of digits. Condition (ii) is included for if $n=r_b(n)$, then condition (iii) holds trivially, and so nothing is surprising. The picture \eqref{picture} in the general case, with the factorizing step omitted, would be as follows.
\begin{align}
n\quad&\xmapsto{\phantom{aaaa}}\quad v(n) \nonumber \\
\uparrow\quad&\phantom{\xmapsto{\phantom{aaaa}}}\quad\phantom{aa}\parallel\nonumber\\
\text{digit reversal in base $b$}&\phantom{\xmapsto{\phantom{aaaa}}}\phantom{aa}\text{same number}\nonumber\\
\downarrow\quad&\phantom{\xmapsto{\phantom{aaaa}}}\quad\phantom{aa}\parallel\nonumber\\
r_b(n)\quad&\xmapsto{\phantom{aaaa}}\quad v(r_b(n)) \nonumber
\end{align}
Only base ten is dealt with in \cite{tsai0,tsai1,tsai2}, but generalizing the definition to a general base is easily done. The sequence of $v$-palindromes in base ten has been created as sequence A338039 in the On-Line Encyclopedia of Integer Sequences \cite{oeis}.

We explain why the $v$-palindromes can be viewed as an analogy to the usual palindromes. Recall the following definition of the usual palindromes.
\begin{definition}[palindrome]
Let $b\geq2$ be an integer. A natural number $n$ is a \emph{palindrome} in base $b$ if $n=r_b(n)$.
\end{definition}
The definition of $v$-palindromes can be obtained from that of the usual palindromes by applying $v$ to the equality $n=r_b(n)$ and then including the conditions (i) and (ii). The reason for including these conditions being as explained earlier. The mere application of $v$ to the equality $n=r_b(n)$ causes palindromes and $v$-palindromes to behave very differently. Although the function $v$ is the specific function as defined above, there is nothing special about it, and it is equally conceivable to use any function $w\colon\mathbb{N}\to\mathbb{C}$ instead in the definition of $v$-palindromes, perhaps calling the defined kind of number $w$-palindromes.

In this article, we prove the existence of a $v$-palindrome for infinitely many bases, and at the end include a section collecting some conjectures on the $v$-palindromes.

\section{Infinitude for base ten.}

We illustrated \eqref{picture} that $198$ is a $v$-palindrome in base ten. That there are infinitely many $v$-palindromes in base ten is shown in \cite{tsai0} by specifically showing that all of the numbers
\begin{equation}\label{eq:1998}
18, 198, 1998,\ldots,
\end{equation}
with any number of nines in the middle, are $v$-palindromes in base ten. Also mentioned in \cite{tsai0} is that all of the numbers
\begin{equation}\label{eq:ten}
18,1818,181818,\ldots,
\end{equation}
with any number of $18$'s repeatedly concatenated, are $v$-palindromes in base ten. In fact the main theorem of \cite{tsai1} is inspired by the sequence \eqref{eq:ten}. Both families \eqref{eq:1998} and \eqref{eq:ten} seem to be ``based on'' $18$. In fact they are subsets of the following more general family.
\begin{theorem}[\cite{tsai3}, Theorem 3]
  If $\rho$ is a palindrome in base ten consisting entirely of the digits $0$ and $1$, then $18\rho$ is a $v$-palindrome in base ten.
\end{theorem}
This theorem relates the usual palindromes with the $v$-palindromes. If we take $\rho$ to be a repunit, then we recover \eqref{eq:1998}. If we take $\rho$ to have alternating digits of $0$ and $1$, then we recover \eqref{eq:ten}. If we take $\rho$ to have only the first and last digits being $1$ and at least one $0$ in between, then we deduce the family
\begin{equation}
1818,18018,180018,\ldots,
\end{equation}
with any number of $0$'s in between two $18$'s, of $v$-palindromes in base ten.

Thus the infinitude of $v$-palindromes in base ten is very well-established. In \S\ref{sec:main}, we shall establish the infinitude of $v$-palindromes for infinitely many bases.

\section{A periodic phenomenon.}\label{sec:aperiodic}

We state the main theorem of \cite{tsai1}, which describes a periodic phenomenon involving $v$-palindromes and repeated concatenations in base ten, for a general base. The proof in \cite{tsai1} is only for base ten, but is easily adapted for a general base. Before that, we first give a notation for repeated concatenations.
\begin{definition}
Suppose that $n=(a_{L-1}\ldots a_1a_0)_b$ is a base $b$ representation and $k\geq1$ is an integer, then we denote the repeated concatenation of the base $b$ digits of $n$ consisting of $k$ copies of $n$ by $n(k)_b$. That is,
\begin{align}
  n(k)_b & =(\underbrace{a_{L-1}\cdots a_1a_0a_{L-1}\cdots a_1a_0\cdots\cdots a_{L-1}\cdots a_1a_0}_\text{$k$ copies of $a_{L-1}\cdots a_1a_0$})_b \nonumber \\
  & =n(1+b^L+\cdots + b^{(k-1)L})=n\cdot\frac{1-b^{Lk}}{1-b^L}. \label{wd}
\end{align}
\end{definition}
For instance, $18(3)_{10}=181818$ and $201(4)_{10}=201201201201$. We can now state the main theorem of \cite{tsai1} for a general base as follows.
\begin{theorem}[{\cite[Theorem 1 for a general base]{tsai1}}]\label{tsai1T}
Let $b\geq2$ be an integer. For every natural number $n$ with $b\nmid n$ and $n\neq r_b(n)$, there exists an integer $\omega\ge1$ such that for all integers $k\ge1$,
\begin{equation}
n(k)_b\in\mathbb{V}_b\quad\text{if and only if}\quad n(k+\omega)_b\in\mathbb{V}_b.
\end{equation}
\end{theorem}
Based on this theorem, we can make the following definitions.
\begin{definition}
The smallest possible $\omega$ in the above theorem is denoted by $\omega_0(n)_b$. If the base $b$ digits of $n$ can be repeatedly concatenated to form a $v$-palindrome in base $b$, i.e., if there exists an integer $k\geq1$ such that $n(k)_b\in \mathbb{V}_b$, then the smallest $k$ is denoted by $c(k)_b$; otherwise we set $c(n)_b=\infty$.
\end{definition}
The sequence of numbers $n$ such that $c(n)_{10}<\infty$ has been created as sequence A338371 in the On-Line Encyclopedia of Integer Sequences \cite{oeis}. Hence there remains the problem of finding $\omega_0(n)_b$ and $c(n)_b$. In fact, \cite{tsai2} completely settles this problem by associating to each $n$ a periodic function $\mathbb{Z}\to\{0,1\}$ which we describe in the next section.

\section{Associated periodic function.}\label{sec:associated}

Fix a base $b\geq2$ and a natural number $n$ with $b\nmid n$ and $n\neq r_b(n)$ throughout this section. To have a clearer picture of the periodic phenomenon illustrated in Theorem \ref{tsai1T}, we define the function $I^n_b\colon \mathbb{N}\to\{0,1\}$ by setting
\begin{equation}
  I^n_b(k)=\begin{cases}
    0\quad\text{if $n(k)_b\notin\mathbb{V}_b$}\\
    1\quad\text{if $n(k)_b\in\mathbb{V}_b$}.
  \end{cases}
\end{equation}
Then $I^n_b$ is a periodic function. It therefore has a unique periodic extension $I^n_b\colon\mathbb{Z}\to\{0,1\}$ which we give the same notation. \cite[Theorem 5.8]{tsai2} says that $I^n_b$ can be expressed as a linear combination, when $b=10$, but the same holds for a general base. We first give notation for certain functions used to form the linear combination.
\begin{definition}
For a natural number $a$, denote by $I_a\colon\mathbb{Z}\to\{0,1\}$ the function defined by
\begin{equation}
  I_a(k)=\begin{cases}
    0\quad\text{if $a\nmid k$}\\
    1\quad\text{if $a\mid k$}.
  \end{cases}
\end{equation}
That is, $I_a$ is the indicator function of $a\mathbb{Z}$ in $\mathbb{Z}$.
\end{definition}
We can now state the linear combination as follows.
\begin{theorem}[\cite{tsai2}, Theorem 5.8 for a general base]
The function $I^n_b$ can be expressed in the form
\begin{equation}\label{eq:expr}
I^n_b=\lambda_1 I_{a_1}+\lambda_2 I_{a_2}+\cdots+\lambda_u I_{a_u},
\end{equation}
where the $u\geq0$, $1\leq a_1<a_2<\cdots<a_u$, and $\lambda_1,\lambda_2,\ldots,\lambda_u\neq0$ are integers. Moreover, the expression is unique.
\end{theorem}
Having expressed the function $I^n_b$ in the form \eqref{eq:expr}, we have the following for finding $\omega_0(n)_b$ and $c(n)_b$.
\begin{theorem}[\cite{tsai1}, Corollaries 6.5 and 7.2 for a general base]
The smallest period $\omega_0(n)_b$ and $c(n)_b$ can be found from the expression \eqref{eq:expr} by
\begin{align}
  \omega_0(n)_b&=\lcm\{a_1,a_2,\ldots,a_u\},\\
  c(n)_b&=\inf\{a_1,a_2,\ldots,a_u\}=\begin{cases}
  \infty\quad\text{if $u=0$}\\
    a_1\quad\text{if $u\geq1$}.
  \end{cases}
\end{align}
\end{theorem}
This infimum is thought of as that in the extended real number system.

We did not say how to express $I^n_b$ in the form \eqref{eq:expr}. A definite procedure for doing this, for $b=10$, is described in \cite{tsai1}, which is easily adapted for a general base.

\section{One implies infinitely many.}

We show, using the ideas in \S\ref{sec:aperiodic} and \S\ref{sec:associated}, that if there exists a $v$-palindrome in base $b$, then there exist infinitely many.

\begin{theorem}\label{thm:inf}
Let $b\geq2$ be an integer. If there exists a $v$-palindrome in base $b$, then there exist infinitely many $v$-palindromes in base $b$.
\end{theorem}
\begin{proof}
Suppose that $n$ is a $v$-palindrome in base $b$. According to \S\ref{sec:associated}, we have the associated function $I^n_b$. That $n$ is a $v$-palindrome in base $b$ means that $I^n_b(1)=1$. Since $I^n_b$ is periodic, say with $\omega$ as a period, we see that
\begin{equation}
I^n_b(1)=I^n_b(1+\omega)=I^n_b(1+2\omega)=\cdots.
\end{equation}
Consequently,
\begin{equation}
n(1)_b,n(1+\omega)_b,n(1+2\omega)_b,\ldots
\end{equation}
are all $v$-palindromes in base $b$.
\end{proof}

\section{Existence of \texorpdfstring{$v$}-palindromes for infinitely many bases.}\label{sec:main}

In this section we show the existence of $v$-palindromes (and therefore infinitely many $v$-palindromes in view of Theorem \ref{thm:inf}) for infinitely many bases. Everything is based on the humble fact that $v(5)=v(6)$. Since $v(n)$ is an additive function, for every integer $t\geq1$ with $(t,30)=1$, $v(5t)=v(6t)$.

Imagine that we have a base $b\geq2$ for which we would like to show that a $v$-palindrome exists. The first simple try would be to look in the two-digit numbers. That is, numbers $(ac)_b=ab+c$, where $1\leq a<c<b$ are integers. By definition, $(ac)_b$ is a $v$-palindrome in base $b$ if and only if $v((ac)_b)=v((ca)_b)$, or equivalently,
\begin{equation}
  v(ab+c)=v(cb+a).
\end{equation}
This would hold if for some integer $t\geq1$ with $(t,30)=1$,
\begin{equation}\label{eq:system}
\begin{cases}
ab+c=5t,\\
cb+a=6t,
\end{cases}
\end{equation}
simply by the observation in the previous paragraph. To summarize, we have shown the following.
\begin{lemma}
Let $b\geq2$ be an integer. If there exists an ordered triple $(a,c,t)$ of positve integers such that $a<c<b$, $(t,30)=1$ and \eqref{eq:system} holds, then the two-digit number $(ac)_b$ is a $v$-palindrome in base $b$. Hence in particular there exists a $v$-palindrome in base $b$.
\end{lemma}
\begin{definition}
We call a triple $(a,c,t)$ in the premise of the above lemma a \emph{permissible triple} for $b$.
\end{definition}
So our strategy is to try to find permissible triples. The system \eqref{eq:system} can be written in matrix from as
\begin{equation}
\begin{pmatrix}
b & 1\\
1 & b
\end{pmatrix}
\begin{pmatrix}
a\\
c
\end{pmatrix}=
t\begin{pmatrix}
5\\
6
\end{pmatrix}.
\end{equation}
Solving we have
\begin{align}
\begin{pmatrix}
a\\
c
\end{pmatrix}&=t\begin{pmatrix}
b & 1\\
1 & b
\end{pmatrix}^{-1}\begin{pmatrix}
5\\
6
\end{pmatrix}=\frac{t}{b^2-1}
\begin{pmatrix}
b & -1\\
-1 & b
\end{pmatrix}
\begin{pmatrix}
5\\
6
\end{pmatrix}\\
&=\frac{t}{b^2-1}
\begin{pmatrix}
5b-6\\
-5+6b
\end{pmatrix}=
\begin{pmatrix}
\frac{t(5b-6)}{b^2-1}\\
\frac{t(-5+6b)}{b^2-1}
\end{pmatrix}.
\end{align}
We write them separately as
\begin{equation}\label{eq:frac}
  a=\frac{t(5b-6)}{b^2-1},\quad c=\frac{t(-5+6b)}{b^2-1},
\end{equation}
from which we also see that $0<a<c$. Hence we have the following lemma.
\begin{lemma}
Let $b\geq2$ be an integer. For every integer $t\geq1$, there exist unique rational numbers $a,c\in\mathbb{Q}$ such that \eqref{eq:system} holds, and they are given by \eqref{eq:frac}. Moreover, $0<a<c$.
\end{lemma}
Hence the only possible permissible triples for $b$ are
\begin{equation}
\left(\frac{t(5b-6)}{b^2-1},\frac{t(-5+6b)}{b^2-1},t\right),
\end{equation}
for an integer $t\geq1$ with $(t,30)=1$. The only missing conditions to fulfill are
\begin{align}
\frac{t(5b-6)}{b^2-1},\frac{t(-5+6b)}{b^2-1}&\in\mathbb{Z},\label{eq:integers}\\
\frac{t(-5+6b)}{b^2-1}&<b.
\end{align}
We write
\begin{align}\label{eq:glance}
\frac{t(5b-6)}{b^2-1}&=\frac{t(5b-6)/(5b-6,b^2-1)}{(b^2-1)/(5b-6,b^2-1)},\\
\frac{t(-5+6b)}{b^2-1}&=\frac{t(-5+6b)/(-5+6b,b^2-1)}{(b^2-1)/(-5+6b,b^2-1)}.
\end{align}
Hence we see that \eqref{eq:integers} holds if and only if $t$ is a multiple of
\begin{equation}
  f(b)=\left[\frac{b^2-1}{(5b-6,b^2-1)},\frac{b^2-1}{(-5+6b,b^2-1)}\right];
\end{equation}
here we also defined the function $f(b)$ for integers $b\geq2$. Hence we have shown the following lemma.
\begin{lemma}
Let $b\geq2$ be an integer. Then the permissible triples of $b$ are precisely the triples
\begin{equation}
  \left(\frac{t(5b-6)}{b^2-1},\frac{t(-5+6b)}{b^2-1},t\right),
\end{equation}
where
\begin{equation}
t\in S(b)=\left\{t\in\mathbb{N}\colon (t,30)=1,\, f(b)\mid t,\, t<\frac{b(b^2-1)}{-5+6b}\right\};
\end{equation}
where we also defined the set-valued function $S(b)$ for integers $b\geq2$.
\end{lemma}
However, the above lemma does not promise that permissible triples exist, i.e., $S(b)\neq\varnothing$. However, we can get the following sufficient condition.
\begin{lemma}
Let $b\geq2$ be an integer. If
\begin{equation}
  (f(b),30)=1,\quad f(b)<\frac{b(b^2-1)}{-5+6b},
\end{equation}
then $f(b)\in S(b)$, and consequently there is a permissible triple for $b$.
\end{lemma}
Since $f(b)\mid b^2-1$, if $(b^2-1,30)=1$ then $(f(b),30)=1$. Hence the above lemma can be weakened to the following.
\begin{lemma}
Let $b\geq2$ be an integer. If
\begin{equation}\label{eq:conditions}
  (b^2-1,30)=1,\quad f(b)<\frac{b(b^2-1)}{-5+6b},
\end{equation}
then $f(b)\in S(b)$, and consequently there is a permissible triple for $b$.
\end{lemma}
We now consider the condition $(b^2-1,30)=1$. It is easily shown that this is equivalent to that both $b\equiv \modd{0} {6}$ and $b\equiv \modd{0,2,3} {5}$. In particular, $b\equiv \modd{0} {30}$ is a sufficient condition. Suppose that $k\geq1$ is an integer, then
\begin{align}
f(30k)&=\left[\frac{(30k)^2-1}{(5(30k)-6,(30k)^2-1)},\frac{(30k)^2-1}{(-5+6(30k),(30k)^2-1)}\right]\\
&=\left[\frac{(30k)^2-1}{(6k-2,11)},\frac{(30k)^2-1}{(5k+2,11)}\right],
\end{align}
where for the second equality we used a property of the greatest common divisor function to simplify. Because of the right inequality in \eqref{eq:conditions}, we want $f(30k)$ to be small. Thus it might be good if we have $(6k-2,11)=(5k+2,11)=11$, which is easily shown to be equivalent to that $k\equiv \modd{4} {11}$. Whence assume that $k\equiv \modd{4} {11}$, then
\begin{equation}
  f(30k)=\frac{(30k)^2-1}{11}.
\end{equation}
On the other hand, the right-hand-side of the right inequality \eqref{eq:conditions} becomes
\begin{equation}
\frac{(30k)((30k)^2-1)}{-5+6(30k)}.
\end{equation}
That $f(30k)$ is strictly less than the above quantity is equivalent to
\begin{equation}
  -5+6(30k)<11(30k),
\end{equation}
which clearly always holds. Hence the above lemma can be further weakened to the following.
\begin{theorem}
Let $k\equiv \modd{4} {11}$ be a positive integer, then
\begin{equation}
  \left(\frac{-6+150k}{11},\frac{-5+180k}{11},\frac{-1+900k^2}{11}\right)
\end{equation}
is a permissible triple for the base $30k$. In particular, the two-digit number
\begin{equation}
\left(\frac{-6+150k}{11},\frac{-5+180k}{11}\right)_{30k}
\end{equation}
is a $v$-palindrome in base $30k$.
\end{theorem}
Hence we have proved the existence of $v$-palindromes for infinitely many bases, summarized as follows.
\begin{corollary}
  If $b\equiv \modd{120} {330}$ is a positive integer, then there exists a $v$-palindrome in base $b$.
\end{corollary}
Hence in particular there is a positive density of bases $b\geq2$ for which a $v$-palindrome exists.

\section{Existence of \texorpdfstring{$v$}-palindromes for all bases?}

While the previous sections showed that $v$-palindromes exist for all bases $b\equiv \modd{120} {330}$, the goal would be to show for all bases $b\geq2$. The proof in the previous section is based on the equality $v(5)=v(6)$. It is conceivable that the same method basing on other common values of $v$ will find other bases $b$ for which a $v$-palindrome exists. For instance, we have
\begin{gather}
v(5)=v(6)=v(8)=v(9),\\
v(7)=v(10)=v(12)=v(18). 
\end{gather}
Perhaps exploiting this method will lead to resolving the existence of $v$-palindromes for all bases. We give the table at the end of the paper of the smallest $v$-palindrome, i.e., $\min(\mathbb{V}_b)$, for the first few bases, calculated using PARI/GP \cite{pari}.

\section{Some conjectures.}

In the short note \cite{tsai0}, three conjectures on $v$-palindromes has been proposed by commentators. They were conjectured for base ten, but the general conjecture for a general base is easily conceived for two of them. We state those two conjectures as follows for a general base.

\begin{conjecture}
Let $b\geq2$ be an integer. There does not exist a prime $v$-palindrome in base $b$.
\end{conjecture}
\begin{conjecture}
Let $b\geq2$ be an integer. There are infinitely many $v$-palindromes $n$ in base $b$ such that both $n$ and $r_b(n)$ are squarefree.
\end{conjecture}
While \cite{pong} provides an exact formula for the number of palindromes up to a given positive integer, the same can be considered for $v$-palindromes, namely the following.
\begin{problem}
Let $b\geq2$ be an integer. Is there a formula for the number of $v$-palindromes in base $b$ up to a given positive integer? If not, how can it be approximated?
\end{problem}
From $199$ till $575$ are $377$ consecutive positive integers each not a $v$-palindrome in base ten. Just as consecutive composite numbers can be aribitrarily long, we make the following conjecture.
\begin{conjecture}
Let $b\geq2$ be an integer. Consecutive positive integers each not a $v$-palindrome in base $b$ can be arbitrarily long.
\end{conjecture}
Just as is the case for any special kind of number, the possible questions which can be asked about $v$-palindromes is uncountable.

\subsection*{Acknowledgements}
The author wish to thank his advisor Professor Kohji Matsumoto, whose initial encouragement led to the work described in this article.

\bibliography{refs}{}

\begin{thebibliography}{1}
\bibitem{pong} Pongsriiam, P., Subwattanachai, K. (2019). Exact formulas for the number of palindromes up to a given positive integer. \textit{Int.\ J.\ Math.\ Comput.\ Sci.} 14(1): 27--46.

\bibitem{oeis} The On-Line Encyclopedia of Integer Sequences. \url{https://oeis.org}

\bibitem{pari} The PARI Group. (2020). PARI/GP, Version 2.13.0. Bordeaux. \url{http://pari.math.u-bordeaux.fr/}.

\bibitem{tsai1} Tsai, D. (2021). A recurring pattern in natural numbers of a certain property. \textit{Integers.} {21}: Paper No.\ A32.

\bibitem{tsai0} Tsai, D. (2018). Natural numbers satisfying an unusual property. \textit{S\={u}gaku Seminar.} 57(11): 35--36 (written in Japanese).

\bibitem{tsai3} Tsai, D. (2021). Repeated concatenations in residue classes. \url{https://arxiv.org/abs/2109.01798}

\bibitem{tsai2} Tsai, D. (2021). The fundamental period of a periodic phenomenon pertaining to $v$-palindromes. \url{https://arxiv.org/abs/2103.00989}
\end{thebibliography}
\bibliographystyle{plain}

\begin{table}[ht]
 \caption{The smallest $v$-palindrome for bases $b\leq 19$.}
 \label{table:indicatorfun}
 \centering
  \begin{tabular}{lll}
   \hline
   $b$ & $\text{$\min(\mathbb{V}_b)$ written in base $10$}$ & $\text{$\min(\mathbb{V}_b)$ written in base $b$}$ \\
   \hline \hline
   $2$ & $175$ & $1,0,1,0,1,1,1,1$ \\
   $3$ & $1280$ & $1,2,0,2,1,0,2$ \\
   $4$ & $6$ & $1,2$ \\
   $5$ & $288$ & $2,1,2,3$ \\
   $6$ & $10$ & $1,4$ \\
   $7$ & $731$ &  $2,0,6,3$\\
   $8$ & $14$ & $1,6$ \\
   $9$ & $93$ & $1,1,3$ \\
   $10$ & $18$ & $1,8$ \\
   $11$ & $135$ & $1,1,3$ \\

$12$ & $22$ & $1,10$\\
$13$ & $63$ & $4,11$\\

$14$ & $26$ & $1,12$\\
 $15$ & $291$ & $1,4,6$\\
 $16$ & $109$ & $6,13$\\
 $17$ & $581$ & $2,0,3$\\
 $18$ & $34$ & $1,16$\\
 $19$ & $144$ & $7,11$ \\
   \hline
  \end{tabular}
\end{table}

\end{document}